\documentclass[a4paper]{amsart}
\usepackage[active]{srcltx}
\usepackage[all]{xy}
\usepackage{subfig}
\usepackage{mathrsfs}
\usepackage{esint}
\usepackage{amsmath}
\usepackage{amssymb,latexsym}
\usepackage{mathrsfs}
\usepackage{graphics}
\usepackage{latexsym}
\usepackage{psfrag}
\usepackage{import}
\usepackage{verbatim}
\usepackage{graphicx}
\usepackage[usenames]{color}
\usepackage{pifont,marvosym}

\theoremstyle{plain}
\newtheorem{lemma}{Lemma}
\newtheorem{theorem}[lemma]{Theorem}

\theoremstyle{definition}

\newtheorem{remark}[lemma]{Remark}

\numberwithin{equation}{section}

\newcommand{\be}{\begin{equation}}
\newcommand{\ee}{\end{equation}}
\DeclareMathOperator{\A}{\mathcal{A}}
\DeclareMathOperator{\V}{\mathcal{V}}
\newcommand{\de}{\partial}
\newcommand{\Ric}{{\rm Ric}}

\newcommand{\K}{{\rm K}}
\newcommand{\cM}{{\mathcal M}}

\newcommand{\Le}{{\mathcal{L}}}

\newcommand{\RP}{\mathbb{RP}}
\newcommand{\bS}{{\mathbb S}}

\newcommand{\R}{\mathbb{R}}

\newcommand{\N}{\mathbb{N}}

\newcommand{\pa}{\partial}

\newcommand{\ve}{\varepsilon}

\newcommand{\cI}{\mathcal{I}}

\def\Om{\Omega}

\def\l{\lambda}

\begin{document}

\title[Rigidity for critical points in the L\'evy-Gromov inequality] {Rigidity for critical points in the L\'evy-Gromov inequality}
\author{Fabio Cavalletti}\thanks{F. Cavalletti:Universit\'a degli Studi di Pavia, Dipartimento di Matematica, email: fabio.cavalletti@unipv.it}
\author{Francesco Maggi}\thanks{F. Maggi: International Centre for Theoretical Physics, Trieste (on leave from UT Austin), email: fmaggi@ictp.it}
\author {Andrea Mondino} \thanks{A. Mondino: University of Warwick,  email: A.Mondino@warwick.ac.uk}
%
%\address{Centro de Giorgi - SNS}
%\email{fabio.cavalletti@sns.it}


\bibliographystyle{plain}

\begin{abstract}
The L\'evy-Gromov inequality states that round spheres have the least isoperimetric profile (normalized by total volume) among Riemannian manifolds with a fixed positive lower bound on the Ricci tensor. In this note we study critical metrics corresponding to the L\'evy-Gromov inequality and prove that, in two-dimensions, this criticality condition is quite rigid, as it characterizes round spheres and projective planes.
\end{abstract}

\maketitle
%\tableofcontents

The {\it isoperimetric problem} in a closed (i.e. compact without boundary) $n$-dimensional Riemannian manifold $(M,g)$ consists in minimizing the area $\A_g(\de\Om)$ of the boundary $\de\Om$ of a region $\Om\subset M$ with given $n$-dimensional volume $\V_{g}(\Omega)$. Minimizers are called {\it isoperimetric regions}, and the minimum value function is called the {\it isoperimetric profile} of $(M,g)$:
\be\label{eq:defIg}
\cI_{(M,g)}(v)=\inf\left\{\A_{g}(\de\Omega)\,:\, \Omega\subset M, \ \frac{\V_{g}(\Omega)}{\V_{g}(M)}=v \right\}, \qquad v \in (0,1)\,.
\ee
A full solution to the isoperimetric problem requires the explicit characterization of its minimizers, and it is thus possible only in highly symmetric ambient spaces. In the case of generic ambient spaces, the best expectation is to obtain some indirect information, for example in the form of explicit bounds on the isoperimetric profile.

This is the spirit of the celebrated {\it L\'evy--Gromov inequality} \cite[Appendix C]{Gro}: if $\Ric_{g}\geq K\,g$ for some constant $K>0$, then
\begin{equation}
  \label{levygromov inq}
  \frac{\cI_{(M,g)}(v)}{\V_g(M)}\ge\frac{\cI_{(S,g_S)}(v)}{\V_g(S)}\qquad\forall v\in(0,1)\,,
\end{equation}
where $(S,g_S)$ is the standard $n$-dimensional sphere with Ricci curvature equal to $K$ (see also \cite{BBG},    \cite{Mil} for the generalization to  the case $K\leq0$ and  diameter bounded above, and \cite{CM} for the extension to non-smooth spaces).  Having in mind the relation between the Euclidean isoperimetric theorem (balls are the only volume-constrained minimizers of perimeter) and Alexandrov's rigidity theorem (balls are the only volume-constrained critical points of perimeter), in this note we ask what can be said about critical points in the variational problem corresponding to the L\'evy-Gromov inequality, and, at least in dimension two, we prove a full rigidity theorem.

Our terminology will be as follows. The {\it L\'evy-Gromov functional} on a Riemannian manifold $(M,g)$ at volume fraction $v\in(0,1)$ is defined as
\begin{equation}\label{eq:defLGFunct}
\Le_v(M,g)=\frac{\cI_{(M,g)}(v)}{\V_g(M)}\,.
\end{equation}
We denote with $\cM_{M}$ the space of Riemannian metrics over $M$ and, given $K\in\R$, we consider the family $\cM_{M,K}$ metrics on $M$ with Ricci tensor bounded below by $K$, and the family $\cM_{M,K,g}$ of metrics in $\cM_{M,K}$ that are conformal to a given metric $g$, i.e. we set
\begin{eqnarray}
\label{eq:defcMK}
\cM_{M,K}&=&\big\{ g \in \cM_{M}\,:\, \Ric_{g}\geq K \, g\big\}\,,
\\ \label{eq:defcM}
\cM_{M,K,g}&=&\big\{\hat{g}_{ij}:=e^{2u}g_{ij}\,:\, u\in C^2(M), \, \Ric_{\hat{g}}\geq K \hat{g}\big\}\,.
\end{eqnarray}
Endowing $\cM_{M}$ with the $C^2$-topology, we notice that both $\cM_{M,K}$ and $\cM_{M,K,g}$ have non-empty boundary. A natural definition of critical point associated to the Levy-Gromov inequality is then the following: we say that $g$ is a \emph{critical isoperimetric metric} (with constant $K$) if $g\in \cM_{M,K}$ and the following holds:
\begin{enumerate}
\item[(i)] if $g$ is an interior point of  $\cM_{M,K}$, then
\[
\frac{d}{dt}\Big|_{t=0}  \Le_{v}(M,g(t))=0
\]
for every $v\in(0,1)$ and $g(t)\in C^{1}((-1,1);\cM_{M,K})$ with $g(0)=g$;
\item[(ii)] if $g$ is a boundary point of $\cM_{M,K}$, then
\[
\frac{d}{dt}\Big|_{t=0^+}  \Le_{v}(M,g(t))\geq 0
\]
for every $v\in(0,1)$ and $g(t)\in C^{1}([0,1);\cM_{M,K})$ with $g(0)=g$.
\end{enumerate}
When, in the above definition, $\cM_{M,K,g}$ is considered in place of $\cM_{M,K}$, we say that $g$ is a \emph{conformally-critical isoperimetric metric}. The question we pose is what degree of rigidity can be expected for conformally critical isoperimetric metrics.

A first remark is that no metric can be conformally-critical with constant $K\le 0$. Indeed, let us recall that if $\hat{g}=e^{2u} g$ for some $u\in C^2(M)$, then
\begin{equation}
\label{eq:confRic}
\Ric_{\hat{g}}=\Ric_{g}-(\Delta u)\, g-(n-2) {\rm Hess}_g\, u +(n-2) (d u \otimes du-|\nabla u|^{2} g),
\end{equation}
where  ${\rm Hess}_g\, u$ denotes the Hessian of $u$ (with respect to the Levi-Civita connection of $g$) and $\Delta_g u=g^{ij} ({\rm Hess}_g\,u)_{ij}$ is the Laplace-Beltrami operator with respect to $g$ applied to $u$. In particular, if we pick $u=\log\l$ for some $\l>0$ and $\Ric_g\ge K\,g$, then $\Ric_{\hat{g}}=\Ric_{g}\ge K\,g=K\,\l^{-2}\,\hat g$. Given that $K\le 0$, we have $\Ric_{\hat g}\ge K\,\hat g$ for every $\l^2\ge 1$. Since $\V_{\hat{g}}(\Om)=\l^n\,V_g(\Om)$ and $\A_{\hat{g}}(\pa\Om)=\l^{n-1}\,\A_g(\pa\Om)$ for every $\Om\subset M$, we also have
\[
\Le_v(M,\hat{g})=\frac{\Le_v(M,g)}\l
\]
and thus, setting $g(t)=(1+t)^2\,g$ for $t>0$, we find $(d\Le_v(M,g(t))/dt)|_{t=0^+}=-\Le_v(M,g)<0$ for every $v\in(0,1)$.

From now on we shall thus take $K>0$. In dimension $n=2$ (where one simply has $\Ric_{g}=\K_{g} \, g$, $\K_{g}$ denoting the Gauss curvature of $g$) it turns out that the apparently very weak notion of conformally-critical isoperimetric metric implies the maximal degree of rigidity one could expect:

\begin{theorem}[Rigidity of conformally-critical isoperimetric metrics in dimension 2]\label{thm:dim2}
If $(M,g)$ is a two-dimensional closed Riemannian manifold and $K>0$, then $g$ is a conformally-critical isoperimetric metric with constant $K$ if and only if $(M,g)$ is either a sphere or the real projective plane with $K_g=K$.
\end{theorem}

We now present the proof of Theorem \ref{thm:dim2}. For the sake of clarity we work in dimension $n$ until the last step of the argument. We also notice that we shall use conformal-criticality only on a sequence of volumes $v_h\to 0^+$, and thus that we end up proving a slightly stronger statement than Theorem \ref{thm:dim2}.

\begin{proof}
  [Proof of Theorem \ref{thm:dim2}] {\it Step one}: We start recalling that since $M$ is compact, by the direct method, for every $v\in(0,1)$ there exists an isoperimetric region $\Om$ with $\V_g(\Om)=v\,\V_g(M)$. By standard density estimates, $\Om$ is an open set of finite perimeter whose topological boundary $\de\Om$ is a closed $(n-1)$-rectifiable set, characterized by the property that $x\in\pa\Om$ if and only if $\V_g(\Om\cap B_r(x))\in(0,\V_g(B_r(x)))$ for every $r>0$. (Here and in the following, $B_r(x)$ stands of course for the geodesic ball of center $x$ and radius $r$ in $M$.) Let us denote by $\Sigma$ the {\it isoperimetric sweep} of $(M,g)$, defined as
  \[
  \Sigma=\bigcup\Big\{\pa\Om:\mbox{$\Om$ is an isoperimetric region in $(M,g)$ for some $v\in(0,1)$}\Big\}\,.
  \]
  In this step we prove that for every $x\in\Sigma$ and every $r>0$ small enough, there exists $u\in C^2_c(B_{2r}(x))$ such that
  \begin{equation}
  \label{x}
  \limsup_{t\to 0^+}\frac{\Le_v(M,\hat{g}^{t,u})-\Le_v(M,g)}t\le -\frac{n-1}{\V_g(M)}\,,
  \end{equation}
  where we have set
  \begin{equation}
    \label{eq:defgtu}
  \hat{g}^{t,u}_{ij}:=e^{2tu} g_{ij}, \qquad |t|<\ve\,.
  \end{equation}
  We first notice that, by the area formula,
  \begin{equation}
  \label{eq:V'A'u}
  \frac{d}{dt}\Big|_{t=0}  \V_{\hat{g}^{t,u}}(\Omega)=n \int_{\Omega} u \, dvol_{g}\,,\qquad \frac{d}{dt}\Big|_{t=0}  \A_{\hat{g}^{t,u}}(\de \Omega)=
 (n-1) \int_{\de \Omega} u \, dvol_{g_{|\de \Omega}}\,,
  \end{equation}
  where $dvol_{g_{|\de \Omega}}$ is the $(n-1)$-dimensional volume form induced by $g$ on $\de\Om$. Similarly, if we let $(\Phi^X_{s})_{s \in (-\ve,\ve)}$ denote the flow with initial velocity given by a smooth vector-field $X$ on $M$, and set $\varphi_{X}:=g(\nu_{\de \Omega}, X)$ (where the inner unit normal $\nu_{\de \Omega}$ to $\Om$ is defined on the reduced boundary of $\Om$, thus $vol_{g_{|\de \Omega}}$-a.e. on $\de\Om$), then, by a classical first variation argument, see \cite[Theorem 17.20]{Mag}, there exists a constant $\lambda\in\R$ such that
  \begin{equation}
  \label{eq:V'A'phi}
  \frac{d}{ds}\Big|_{s=0}  \V_{g}(\Phi^X_{s}(\Omega))= - \int_{\de \Omega} \varphi_{X} \, dvol_{g_{|\de\Omega}}\,,
  \qquad
  \frac{d}{dt}\Big|_{t=0}  \A_{g}(\Phi^X_{s}(\de \Omega))=- \lambda \int_{\de \Omega} \, \varphi_{X}\, dvol_{g_{|\de \Omega}}\,.
  \end{equation}
  (The constant $\lambda$ is the (distributional) mean curvature of $\de \Omega$ computed in the metric $g$ with respect to inner normal $\nu_{\de \Omega}$.)
  The combination of \eqref{eq:V'A'u} and  \eqref{eq:V'A'phi} thus gives
\begin{eqnarray}
\qquad \V_{\hat{g}^{t,u}}(\Phi^X_{s}(\Omega))&=&  \V_{g}(\Omega)+  n t \int_{\Omega} u \, dvol_{g} - s    \int_{\de \Omega} \varphi_{X} \, dvol_{g_{|\de\Omega}} + O(t^{2})+O(s^{2}), \label{eq:Vst} \\
\qquad \A_{\hat{g}^{t,u}}(\Phi^{X}_{s}(\de \Omega))&=&  \A_{g}(\de \Omega)+  (n-1) t \int_{\de \Omega} u \, dvol_{g_{|\de \Omega}} - s \lambda   \int_{\de \Omega} \varphi_{X}  \, dvol_{g_{|\de\Omega}} + O(t^{2})+O(s^{2})\,. \label{eq:Ast}
\end{eqnarray}
%where $O(\cdot):\R\to \R$ denotes a general smooth function possibly depending on $\Omega$, $g$, $u$ and $X$ satisfying
%\[
%\limsup_{t\to 0^{+}} \frac{|O(t)|}t<\infty\,.
%\]
Let us now fix $x\in\pa\Om$ for some isoperimetric region $\Om$. For every $r>0$ we can find a smooth vector field $X$ supported in the geodesic ball $B_r(x)$ such that
\[
\int_{\de \Omega} \varphi_{X} \, dvol_{g_{|\de\Omega}}=1\,,
\]
(see \cite[Lemma 17.21]{Mag}). Moreover, if $r$ is small enough, then we can pick $u\in C^2_c(B_{2r}(x))$ such that
\[
\int_{\de \Omega} u \, dvol_{g_{|\de\Omega}}=-1\,,
\qquad \int_{\Omega} u \, dvol_{g}=0\,,\qquad\int_M u\,dvol_{g}=0\,.
\]
Indeed, $B_r(x)\cap\de\Omega$ has positive area, thus there exists $v\in C^0_c(B_r(x)\cap\de\Om)$ with $\int_{\de \Omega} v \, dvol_{g_{|\de\Omega}}<0$. We can thus construct $w_1\in C^2_c(B_r(x))$, $w_2\in C^2_c(\Omega\cap B_{2r}(x)\setminus \overline{B_r(x)})$ and  $w_3\in C^2_c(B_{2r}(x)\setminus\overline{\Om\cup B_r(x)})$ in such a way that
\[
\int_{\de \Omega} w_1 \, dvol_{g_{|\de\Omega}}=-1\,,\qquad\int_M w_2dvol_{g}=-\int_\Om w_1dvol_{g}\qquad\int_Mw_3dvol_{g}=-\int_{M\setminus\Om}w_1dvol_{g}\,,
\]
and then set $u=w_1+w_2+w_3$. We now apply \eqref{eq:Vst} and \eqref{eq:Ast} with these choices of $u$ and $X$, to find
\begin{eqnarray*}
\V_{\hat{g}^{t,u}}(M)&=&  \V_{g}(M)+ O(t^{2})+O(s^{2})\,,
\\
\V_{\hat{g}^{t,u}}(\Phi^X_{s}(\Omega))&=&\V_{g}(\Omega)- s+ O(t^{2})+O(s^{2})\,,
\\
\A_{\hat{g}^{t,u}}(\Phi^{X}_{s}(\de \Omega))&=&  \A_{g}(\de \Omega)- (n-1) t  - s \lambda   + O(t^{2})+O(s^{2})\,.
\end{eqnarray*}
Let us consider the function $F\in C^2((-\ve,\ve)\times(-\ve,\ve))$ defined by
\[
F(s,t)=\frac{\V_{\hat{g}^{t,u}}(\Phi^X_{s}(\Omega))}{\V_{\hat{g}^{t,u}}(M)}\qquad |t|,|s|<\ve\,.
\]
Since $F(0,0)=v$ and $\pa F/\pa s(0,0)=-1/\V_g(M)$, up to decrease the value of $\ve$, there exists a $C^2$-function $s=s(t)$ such that $F(s(t),t)=v$ for every $|t|<\ve$, i.e.
\[
\frac{\V_{\hat{g}^{t,u}}(\Phi^X_{s(t)}(\Omega))}{\V_{\hat{g}^{t,u}}(M)}=v\qquad\forall |t|<\ve\,.
\]
Moreover, $\pa F/\pa t(0,0)=0$ implies $s'(0)=0$, and thus $s(t)=O(t^2)$. Hence,
\[
\frac{\A_{\hat{g}^{t,u}}(\Phi^{X}_{s(t)}(\de \Omega))}{\V_{\hat{g}^{t,u}}(M)}=
\frac{\A_{g}(\de \Omega)}{\V_g(M)}- \frac{(n-1)}{\V_g(M)}\, t + O(t^{2})\,,\qquad\mbox{as $t\to 0$}\,,
\]
so that
\[
\cI_{(M,\hat{g}^{t,u})}(v)\le \cI_{(M,g)}(v)- \frac{(n-1)}{\V_g(M)}\, t + O(t^{2})\qquad\mbox{as $t\to 0$}\,,
\]
and \eqref{x} is proved.

\medskip

\noindent {\it Step two}: Now assuming that $g$ is a conformally critical isoperimetric metric with constant $K>0$, we show that for every $x\in\Sigma$ (the isoperimetric sweep of $M$), there exists $\xi\in T_xM$ such that
\[
\Ric_{g,x}(\xi,\xi)=K\,g_x(\xi,\xi)\,.
\]
Indeed, if this is not the case, then we can find an isoperimetric region $\Om$ and $x\in\pa\Om$ such that
\begin{equation}
  \label{z}
  \Ric_{g,y}(\xi,\xi)>K\,g_y(\xi,\xi)\,,\qquad\forall \xi\in T_yM, \, y\in B_{2r}(x)\,.
\end{equation}
Depending on $x$ and $\Om$, we pick $r$, $X$ and $u$ as in step one. Recall that, in step one, we constructed $u$ so that it was supported in $B_{2r}(x)$. Therefore, by \eqref{z}, we can entail that for every $|t|<\ve$
\[
\Ric_{\hat{g}^{t,u}}\ge K\,\hat{g}^{t,u}\qquad\mbox{on }M\,.
\]
By definition of conformally-critical isoperimetric metric we find a contradiction with \eqref{x}.

\medskip

\noindent {\it Step three}: We now let $n=2$. By step two, $K_g\equiv K$ on the closure of the isoperimetric sweep of $M$. However, it is well-known (see for example \cite{MorgJoh,NardAGA,Druet,MoNa}) that if $\{\Om_h\}_{h\in\N}$ is a sequence of isoperimetric regions corresponding to volume fractions $v_h\to 0^+$ as $h\to\infty$, then $\{\Om_h\}_{h\in\N}$ converges in Hausdorff distance to a point $x$ such that
\[
K_g(x)=\max_M\,K_g\,.
\]
By continuity of $K_g$ we thus conclude that $K=\max_M K_g$, and thus $K_g$ is constantly equal to $K$ on $M$. We have thus proved that if $g$ is conformally-critical with constant $K$, then $K_g\equiv K$, and thus, since $K>0$, that either $(M,g)$ is the sphere or the real projective plane.

\medskip

\noindent {\it Step four}: We are now left to show that both the sphere and the real projective plane are conformally-critical. In the case of the sphere this is immediate from the Levy-Gromov inequality, so that we are left to check the case of the real projective plane.

Without loss of generality we consider the standard projective plane $\RP^{2}$ endowed with the metric $g_{0}$ of constant curvature $K=1$ defined as the quotient of the round sphere $(\bS^{2}, \tilde{g}_{0})$ of unit radius in $\R^{3}$ under the antipodal equivalence relation. We denote by $\Pi: \bS^{2} \to \RP^{2}$ the projection map.

Let us first recall that on a general compact Riemannian surface $(M^{2},g)$ without boundary,  just by considering the complement of each competitor, the isoperimetric profile $\cI_{(M^{2}, g)}$ is symmetric with respect to $v=1/2$. In particular, we shall restrict $v$ to the range $v \in (0,1/2]$.
Moreover, by direct methods and first variation arguments, for every $v$ there exist isoperimetric regions which are necessarily bounded by finitely many curves with constant geodesic curvature.

By direct computation, in the case when $(M^2,g)=(\RP^{2}, g_{0})$ and $v\in [0,1/2]$, isoperimetric regions are metric balls which lift into $\bS^{2}$ as pairs of antipodal spherical cups, each spherical cup having volume $2 v \pi$ in $\bS^2$.

Now assume by contradiction that there exists  $v_{0}\in (0, 1/2]$ and a curve $g_{(\cdot)}: [0,1] \to \cM_{\RP^{2}, 1}$ starting from the round metric $g_{0}$ on $\RP^{2}$, such that
\begin{equation*}%\label{eq:contrRP2}
\limsup_{t \to 0^+} \frac{\Le_{v_{0}} (\RP^{2}, g_{t})- \Le_{v_{0}} (\RP^{2}, g_{0}) } {t} <0.
\end{equation*}
Thus we can find $t_n\to 0^+$ as $n\to\infty$ and isoperimetric regions $\Omega_{t_{n}}$ in  $(\RP^{2}, g_{t_n})$  such that
\begin{equation}\label{eq:contrRP2Om}
\limsup_{n\to\infty} \frac{1}{t_n} \left[ \frac{ \A_{g_{t_n}} ( \de \Omega_{t_n})} {\V_{g_{t_n}} (\RP^{2})} -  \frac{\A_{g_{0}} ( \de \Omega_{0})} {2\pi}   \right] < 0, \quad  \frac{\V_{g_{t_n}} (\Omega_{t_n})}{\V_{g_{t}} (\RP^{2})}=v_{0}\qquad\forall n\in\N\,,
\end{equation}
where $\Omega_{0}\subset \RP^{2}$ is a metric ball in metric $g_{0}$ with $\V_{g_{0}}(\Omega_{0})= 2 v_{0} \pi$. Up to extracting a subsequence and up to translations, by standard density estimates, one can assume that $\Omega_{t_{n}}$ converges to $\Omega_{0}$ in Hausdorff distance with respect to the metric $g_{0}$. (In fact, the convergence is smooth, but this is not needed here.)

Consider now the lifted metrics on $\bS^{2}$ defined by $\tilde{g}_{t}:=\Pi^{*}(g_{t})$ and observe that  $\tilde{g}_{t}\in \cM_{\bS^{2}, 1}$, as  $\tilde{g}_{t}$ is locally isometric to $g_{t}$. Moreover, by construction, $\tilde{g}_{t}$ is invariant under the antipodal map and $\V_{\tilde{g}_{t}} (\bS^{2})=2 \V_{{g}_{t}} (\RP^{2})$. Since $v_{0} \in (0, 1/2]$ and $\Omega_{t_{n}}$ is Hausdorff close to $\Om_0$, the lifted set $\tilde{\Omega}_{t_{n}}:=\Pi^{-1}(\Omega_{t_{n}}) \subset \bS^{2}$ can be written as  $
\tilde{\Omega}_{t_{n}}=\tilde{\Omega}_{t_{n}}^{1}\cup  \tilde{\Omega}_{t_{n}}^{2}$ where  $\tilde{\Omega}_{t_{n}}^{1}, \tilde{\Omega}_{t_{n}}^{2} \subset \bS^{2}$ are $\tilde{g}_{t}$-isometric sets at positive Hausdorff distance. In particular
\begin{equation}\label{eq:AVtOm}
\frac{\V_{\tilde{g}_{t_{n}}} (\tilde{\Omega}_{t_{n}}^{1})}{ \V_{\tilde{g}_{t_{n}}} (\bS^{2})}= \frac{\V_{\tilde{g}_{t_{n}}} (\tilde{\Omega}_{t_{n}}^{2})}{ \V_{\tilde{g}_{t_{n}}} (\bS^{2})} = \frac{1}{2} \frac{\V_{{g}_{t_{n}}} ({\Omega}_{t_{n}})}  { \V_{{g}_{t_{n}}} (\RP^{2})}=\frac{v_{0}}{2},  \quad \A_{\tilde{g}_{t_{n}}} (\de \tilde{\Omega}_{t_{n}}^{1})= \A_{\tilde{g}_{t_{n}}} (\de \tilde{\Omega}_{t_{n}}^{2})=\A_{{g}_{t_{n}}} ({ \de \Omega}_{t_{n}}).
\end{equation}
Notice that for $t=0$ one has that $\tilde{\Omega}_{0}:= \Pi^{-1}(\Omega_{0})$ can be written as $\tilde{\Omega}_{0}=\tilde{\Omega}_{0}^{1}\cup \tilde{\Omega}_{0}^{2}$, where  $\tilde{\Omega}_{0}^{1}$ and $\tilde{\Omega}_{0}^{2}$ are antipodal  spherical caps with $\V_{\tilde{g}_{0}}(\tilde{\Omega}_{0}^{1})= \V_{\tilde{g}_{0}}(\tilde{\Omega}_{0}^{2})=2 \pi v_{0}$. Note that such spherical caps are disjoint and isoperimetric for their own volume in $(\bS^{2}, \tilde{g}_{0})$.

The combination of  \eqref{eq:contrRP2Om} and \eqref{eq:AVtOm} then yields
\begin{align*}
\liminf_{t \to 0^+} \frac{\Le_{v_{0}/2} (\bS^{2}, \tilde{g}_{t})- \Le_{v_{0}/2} (\bS^{2}, \tilde{g}_{0}) } {t} & \leq    \limsup_{n \to \infty} \frac{1}{t_{n}} \left[ \frac{ \A_{\tilde{g}_{t_{n}}} ( \de \tilde{\Omega}_{t_{n}}^{1} )} {\V_{\tilde{g}_{t_{n}}} (\bS^{2})} -  \frac{\A_{\tilde{g}_{0}} ( \de \tilde{\Omega}_{0}^{1})} {4\pi}   \right] \nonumber \\
& = \frac{1}{2}  \limsup_{n \to \infty}  \frac{1}{t_{n}} \left[ \frac{ \A_{g_{t_{n}}} ( \de \Omega_{t_{n}})} {\V_{g_{t_{n}}} (\RP^{2})} -  \frac{\A_{g_{0}} ( \de \Omega_{0})} {2\pi}   \right] < 0,
\end{align*}
contradicting the classical Levy-Gromov inequality for $v=v_{0}/2$ and $K=1$. The proof of step four and then of Theorem \ref{thm:dim2} is thus complete.
\end{proof}

\begin{remark}
  {\rm The above argument actually shows more than what is claimed in Theorem \ref{thm:dim2}, and namely that, if $n=2$ and $g$ is conformally-critical for $\Le_v$ {\it just} for a sequence of values $v=v_h\to 0^+$ as $h\to\infty$, then $K_g$ is constant. Similarly, in step four, we have proved that $(\RP^2,g_0)$ is a critical, and not just conformally critical.}
\end{remark}

\section*{Acknowledgements} This work has been supported by the NSF DMS Grant No. 1265910. Part of the work has been developed while A. M. was in residence at the Mathematical Science Research Institute in Berkeley, California, during Spring 2016, supported by NSF DMS Grant No. 1440140.

\end{document}